\documentclass[a4paper, 10pt, leqno,final]{article}

\usepackage[T1]{fontenc}		     
\usepackage[utf8]{inputenc}			 
\usepackage[english]{babel}

\usepackage{mathrsfs}		
\usepackage{eucal}			

\usepackage{braket}			
\usepackage[all,pdf]{xy}		

\usepackage{amsmath}
\usepackage{amssymb}
\usepackage{asymptote}
\usepackage{amsthm}

	\theoremstyle{plain}
				
						\newtheorem{thm}{Theorem}[section]
		\newtheorem{cor}[thm]{Corollary}	
			
		\newtheorem{lem}[thm]{Lemma}		
				\newtheorem{prop}[thm]{Proposition}

	\theoremstyle{definition}
		\newtheorem{defn}[thm]{Definition}			
		\newtheorem{ex}[thm]{Example}		
			\theoremstyle{remark}
		\newtheorem{rem}[thm]{Remark}

\numberwithin{equation}{section}	

\usepackage{mathrsfs}		
\usepackage{eucal}			
\usepackage{braket}			
\usepackage[all,pdf]{xy}		
\usepackage{mparhack}		
\usepackage{relsize}			
\usepackage{a4wide}		
\usepackage{booktabs}		
\usepackage{multirow}		
\usepackage{caption}		
\usepackage{rotating}
\usepackage{subfig}
\usepackage{varioref}		
\usepackage{footmisc}

\usepackage{graphicx}		
\usepackage{epsfig}

\usepackage{enumerate}		
\newcommand{\tep}{t_{\varepsilon}}
	
\newcommand{\N}{\mathbb{N}}	
\newcommand{\R}{\mathbb{R}}	
\newcommand{\Z}{\mathbb{Z}}
\newcommand{\T}{\mathbb{T}}
\newcommand{\C}{\mathbb{C}}	
\renewcommand{\O}{\mathrm{O}}		
		
\renewcommand{\P}{\mathscr{P}}
\newcommand{\U}{\mathbb{U}}	
\newcommand{\UU}{\mathrm{U}}		
\newcommand{\XY}{\mathbb{X}}

\renewcommand{\L}{L}	
\newcommand{\Gr}{\mathrm{Gr}}

\newcommand{\abs}[1]{\left\lvert #1 \right\rvert}	
\newcommand{\trasp}[1]{{#1}^\mathsf{T}}

\newcommand{\Graph}{\mathrm{Gr}}
\newcommand{\Bsym}{\mathrm{B_{sym}}}
\newcommand{\Sp}{\mathrm{Sp}}

\newcommand{\GL}{\mathrm{GL}}
\newcommand{\Id}{I}
\newcommand{\eps}{\varepsilon}

\newcommand{\Index}{\mathrm{n_-\,}}		
\newcommand{\Coindex}{\mathrm{n_+\,}}		
\newcommand{\nullity}{\mathrm{n_0\,}}		
\newcommand{\sign}{\mathrm{sign\,}}	 
\newcommand{\iCZ}{\iota^{\textup{\tiny{CZ}}}}
\newcommand{\iCLM}{\iota^{\textup{\tiny{CLM}}}}     
\newcommand{\iRS}{\iota^{\textup{\tiny{RS}}}}   
\newcommand{\iMor}{\iota^{\textup{\tiny{}}}} 

		\renewcommand{\leq}{\leqslant}
\renewcommand{\geq}{\geqslant}

\newcommand{\Ueff}{U_k}

\usepackage{bm}

\DeclareMathOperator{\diag}{diag}		
\DeclareMathOperator{\sgn}{sgn}		

\title{Keplerian orbits through the Conley-Zehnder index}
\author{Henry Kavle, Daniel Offin,  Alessandro Portaluri
\thanks{Authors are partially supported Prin 2015 ``Variational methods, with applications to
problems in mathematical physics and geometry” No.~$\mathrm{2015KB9WPT\_001}$.}  }
\date{\today}

\date{\today}
\begin{document}
 \maketitle

\begin{abstract}
It was discovered by Gordon \cite{Gor77} that Keplerian ellipses in the plane are minimizers of the Lagrangian action and spectrally stable as periodic  points of the associated Hamiltonian flow. The aim of this note is  to give a direct proof of these results already proved by authors in \cite{HS10, HLS14} through a self-contained and explicit  computation of the Conley-Zehnder index through crossing forms in the Lagrangian setting.

The techniques developed in this paper can be used to investigate the higher dimensional case of Keplerian ellipses, where the classical variational proof no longer applies. 
\vskip0.2truecm
\begin{center}
\textbf{In memory of our friend Florin Diacu}
\end{center}
\end{abstract}


\section*{Introduction}\label{sec:intro}

In the remarkable paper \cite{Gor77}, Gordon was able to apply the Tonelli direct method in Calculus of Variations for the Lagrangian action functional of the gravitational central force problem, by  proving that the infimum of the Lagrangian action functional on the space of loops in the plane avoiding the origin and having non-vanishing winding number about the origin is realized by the Keplerian orbits (ellipses), including the limiting case of the elliptic collision-ejection orbit which passes through the origin. Excluding the latter, these all have winding numbers $1$ (for the direct orbits) or $-1$ (for the retrograde orbits). 

The main difficulties addressed by author in the aforementioned paper are  due to the lack of compactness. The non-compactness arises because of the unboundedness of  the configuration space as well as the presence of the singularity. The non-compactness due to the  unboundedness of the configuration space can be cured by properly defining the class of loops for minimizing the action. This class of loops was defined by introducing a sort of tied  condition in terms of the winding number. In more abstract terms, the infimum of the action functional in the two of the infinitely many (labelled by the winding number) path connected components of the loop space of the plane with the origin removed having winding number $\pm 1$, is attained precisely on  the Keplerian ellipses (other than the elliptic ejection-collision solutions). We observe that all elliptic orbits with fixed period $T$ have the same Lagrangian action. 

Starting from the aforementioned seminal paper of Gordon dozens of papers using methods from Calculus of Variations for weakly singular Lagrangian problems were published  in the last decades. 

In \cite{HS10}, authors were interested in studying  the relation between the Morse index and the stability for the elliptic Lagrangian solutions of the three body problem. In particular, in the first half of this paper, starting from Gordon's theorem  and by using an index theory of periodic solutions of Hamiltonian system and in particular the Bott-type iteration formula, the authors were able to get a stability criterion for the elliptic Keplerian orbits as well as compute the Morse index of all of their iterations. 

This paper is directed towards a twofold aim: to recover Gordon results through the use of the Conley-Zehnder index, and to recover the results on the Keplerian ellipses given by authors in \cite{HS10} in a maybe more direct way through crossing forms in the Lagrangian setting. 

It is interesting to observe that through the approach  developed in this paper, it is possible to explicitly compute the index  properties of closed Keplerian orbits for surfaces of constant Gaussian curvature among others. We conclude by observing that,  Gordon's theorem breaks down if the dimension of the configuration space is bigger than 2, due to the fact that the loop space of the Euclidean $n$-dimensional space (for $n \geq 3$) with the origin removed is path connected, contrary to what happens with our approach.

The paper is organized as follows:

\tableofcontents


\subsection*{Acknowledgements}
The third name author wishes to thank all faculties and staff at the Queen's University (Kingston)  for providing excellent working conditions during his stay and especially his wife, Annalisa, that  has been extremely supportive of him throughout this entire period  and has made countless sacrifices to help him  getting to this point.


\subsection*{Notation}

For the sake of the reader, we introduce some notation that we shall use throughout the paper. 
\begin{itemize}
\item[-] We denote by $\R^+\:=(0,\infty)$ the positive real numbers. The symbol $\langle \cdot, \cdot \rangle_k$ (or just $\langle \cdot, \cdot \rangle$ if no confusion can arise)  denotes the Euclidean product in $\R^k$ and $\abs{\cdot}$ denotes the (Euclidean) norm.  $\Id_k$ denotes the identity matrix in $\R^k$; as shorthand, we use just the symbol $\Id$. By $\U\subset \C$ we denote the  unit circle (centered at $0$) of the complex plane.
\item[-] $\XY\:=\R^2\setminus\{0\}$ denotes the configuration space, $T\XY$ its tangent bundle, or state space, and $T^*\XY$ its cotangent bundle, or phase space. $\Lambda \XY$ denotes the Hilbert manifold of loops of length $T$ on $\XY$ having Sobolev regularity $H^1$. $\T\:=\R/T\Z$  denotes the unit circle of length $T$.
\item[-]  $\mathcal P_{T}(2n)$ is the set of continuous symplectic maps $\psi:[0,T] \to \Sp(2n)$ such that $\psi(0)=\Id$, as defined in Equation \eqref{eq:pt}.
\item[-]  $\Bsym(V)$ is the set of symmetric bilinear forms on the vector space $V$. For $B \in \Bsym(V)$, we denote by $\sigma(B)$ its spectrum and by  $\Index(B)$, $\Coindex(B)$ and $\nullity(B)$, its {\em index\/} (total number of negative eigenvalues), its {\em coindex\/} (total number of  positive eigenalues) and its nullity  (dimension of the kernel), respectively. \\
 The {\em signature\/} of $B$ is defined by $\sign(B)\:=\Coindex(B)-\Index(B)$. 
\end{itemize}


	\section{Variational and Geometrical framework } \label{sec:variational}

The aim of this section is to briefly describe the problem in its appropriate variational setting and to discuss some basic properties of the bounded non-colliding motions  that we shall use for proving our main results.


\subsection{Variational setting for the Kepler problem}

We consider the gravitational force interaction between two point particles (or bodies) in the Euclidean plane having  masses $m_1, m_2 \in \R^+$. It is well-known that the motion of the two bodies interacting through the gravitational potential is mathematically equivalent to the motion of a single body with a {\em reduced mass\/} equal to 
\begin{equation}\label{eq:reduced-mass}
 \mu\:= \dfrac{m_1m_2}{m_1+m_2}
 \end{equation}
 that is acted on by an (external) attracting gravitational central force (pointing  toward the origin). So,  we let $\XY\:=\R^2\setminus\{0\}$ be the {\em configuration space\/}.  The elements of the {\em state space\/} (namely the tangent bundle $T\XY \cong \XY\times \R^2$) are denoted by $(q,v)$ where $q \in \XY$ and $v \in T_q \XY$. We denote by $U: \XY \to \R$ the  {\em Keplerian potential (function)\/}, which is defined as follows 
\begin{equation}\label{eq:potential-function}
	U(q)\:= \dfrac{m}{\abs{q}}
\end{equation}
for $m\:=Gm_1m_2$ and $G \in \R^+$ denoting the {\em gravitational constant\/}.  We now consider the Lagrangian $L:T\XY\to \R$ defined by 
\begin{equation}\label{eq:Lagrangian}
	L(q, v)=K(v)+U(q) \qquad \textrm{ where } K(v)\:= \dfrac12\mu\abs{v}^2
\end{equation}
where $\mu$ is given in Equation~\eqref{eq:reduced-mass} and  $U$ is given in  Equation~\eqref{eq:potential-function} and we observe that the potential $U$ is a  positively homogeneous function of degree $-1$.
	
We denote by $T^*\XY$ the {\em phase space\/} (i.e. the cotangent bundle of $\XY$). Elements of $T^*\XY$ are denoted by $z=(p,q)$ where $ q \in \XY$ and $p \in T_q^*\XY$. Since the Lagrangian function $L$ defined in Equation~\eqref{eq:Lagrangian} is fiberwise   $\mathscr C^2$-convex (being quadratic in the velocity $v$), then  the Legendre transformation
		\[
			\mathscr{L}_L: T\XY \to T^*\XY, \qquad (q,v)\mapsto \big(D_v L(q,v), q\big),
		\]
	is a smooth (local) diffeomorphism. The Fenchel transform of $L$ is the autonomous smooth Hamiltonian on $T^*\XY$  defined by 
		\[
			H(p,q) \:= \max_{v \in T_q\XY} \big(p[v]-L(q,v)\big) = p\big[v(p,q)\big]- L\big(q,v(p,q)\big),
		\]
	where $\big(q,v(p,q)\big)=\mathscr{L}^{-1}_L(p,q)$. 

Given $T>0$, we denote by  $\T\:=\R/T\Z$  the circle of length $T$ and 
	we denote by $\Lambda \XY\:=H^1(\T, \XY)$ the Hilbert manifold of $T$-periodic loops in $\XY$ having Sobolev regularity $H^1$ with respect to the scalar product 
\begin{equation}\label{eq:scalar-product-h1}
	\langle \gamma_1, \gamma_2\rangle_{H^1}= \int_0^T \big[ \langle \gamma_1'(t), \gamma_2'(t) \rangle + \langle \gamma_1(t), \gamma_2(t) \rangle\big]\, dt. 
\end{equation}
We consider the \emph{Lagrangian action functional} $\mathbb A : \Lambda\XY \to \R$  given  by 
		\[
			\mathbb A(\gamma) \:= \int_0^T  L\big(\gamma(t),\dot \gamma(t)\big)\,dt.
		\]
It is well known  that the Lagrangian action  $\mathbb A$ is of class $\mathscr{C}^2$ (and further, is actually smooth). By a straightforward calculation of  the first variation and up to some standard regularity arguments, it follows that the critical points of $\mathbb A$ are $\mathscr C^2$ solutions  of the Euler-Lagrange equation 
		\begin{equation}\label{eq:kepler-one-center}
	\mu \ddot \gamma(t) = -\dfrac{m}{|\gamma(t)|^3}\gamma(t), \qquad t \in (0,T) 
\end{equation}
Given a classical solution $\gamma$ of Equation~\eqref{eq:kepler-one-center}, the second variation of $\mathbb A$ is represent by the {\em index form\/} given by 
		\begin{equation}\label{eq:secondvariation}
			d^2 \mathbb A(\gamma)[\xi,\eta] = \int_0^T  \big[ \langle P\dot \xi(t),\dot \eta(t)\rangle + \langle R(t) \xi(t),  \eta(t)\rangle\big] \,dt \qquad \forall\, \xi, \eta\in \Lambda \XY
		\end{equation}
	where $P \:=\mu\, \Id $ and  $R(t) \:= D_{qq} \L\big(\gamma(t), \dot \gamma(t)\big)$. In particular, it is a continuous symmetric bilinear form on the Hilbert space $W$ consisting of the $T$-periodic $H^1$ sections $\xi $ of $\gamma^*(T\XY)$. Since the Lagrangian $L$ is exactly quadratic with respect to $v$, it follows that the index form is a compact perturbation of the form 
	\[
	(\xi,\eta) \mapsto \int_0^T \big[\langle P \dot \xi(t), \dot \eta(t) \rangle + \langle \xi(t), \eta(t) \rangle\big]\,dt 
	\] 
	which is coercive on $W$ and hence it is an essentially positive Fredholm quadratic form. 	In particular, its Morse index (i.e. the maximal dimension of the subspace of $W$ such that the restriction of the index form is negative definite), is finite. Thus, we introduce the following definition: 
	\begin{defn}\label{def:morse-gamma}
	Let $\gamma$ be a critical point of $\mathbb A$. We define the {\em Morse index of $\gamma$\/} as the Morse index of $D^2\mathbb A(\gamma)$, the second Frechét derivative of $\mathbb A$ at $\gamma$.
	\end{defn}

%
%
%
%

\subsection{Geometrical properties of Keplerian orbits}\label{subsec:reduction}

By changing to polar coordinates $(r, \vartheta)$  in the configuration space  $\XY$,  we observe that the Lagrangian $L$ defined in Equation~\eqref{eq:Lagrangian} along the smooth curve $t \mapsto\gamma(t)\:=\big(r(t), \vartheta(t)\big)$ reduces to 
\begin{equation}\label{eq:EL-polar-Kepler}
L\big(r,\vartheta, \dot r, \dot \vartheta\big)= \dfrac12 \mu \big(\dot r^2+ r^2\dot \vartheta^2 \big)+ U\big(r\big).
\end{equation}
Thus the Euler-Lagrange  given in Equation~\eqref{eq:kepler-one-center} fits into the following 
\begin{equation}\label{eq:EL-polar}
\begin{cases}
	\mu \ddot r - \mu r\dot \vartheta^2+\dfrac{m}{r^2}=0 &\textrm{ on } \  [0,  T]\\
	\dfrac{d}{dt}\big(\mu r^2\dot{\vartheta}\big)=0.&
\end{cases}	
\end{equation}
We refer to the first (respectively  second) differential equation in Equation~\eqref{eq:EL-polar} as the {\em radial Kepler\/} (respectively  {\em transversal Kepler\/}) {\em equation\/}.  By a direct integration in the second equation, we directly  get 
\begin{equation}\label{eq:ang-mom}
\dot \vartheta= \dfrac{k}{\mu r^2}, \qquad k \in \R^+
\end{equation}
where $k$ denotes the modulus of the {\em angular momentum\/} (which is in fact a conserved quantity).

A second constant of motion is given by the energy $h$. Energy is constant because there are no external forces acting on the reduced body; hence the Lagrangian is time independent.
By substituting  $\dot \vartheta$ given in Equation~\eqref{eq:ang-mom} in the radial Kepler equation, we get 
\begin{equation}\label{eq:radial-kepler}
\mu \ddot r - \dfrac{k^2}{\mu r^3}+\dfrac{m}{r^2}=0 \qquad \textrm{ on }  [0,  T].
\end{equation}
So, we define the {\em effective potential energy \/} as 
\begin{equation}\label{eq:pot-effective}
	\Ueff(r)\:= \dfrac{k^2}{2\mu r^2}-\dfrac{m}{r}
\end{equation}
and we observe that Equation~\eqref{eq:radial-kepler} is nothing but the equation of motion of a particle moving on the line in the force field generated by the potential function $\Ueff$. 
By the {\em energy conservation law\/} we get that the energy level $h$ is determined by 
\begin{equation}\label{eq:energy-level}
	\dfrac12\mu \dot r^2+ \dfrac{k^2}{2\mu r^2} -\dfrac{m}{r}=h.
\end{equation} 
Let us now introduce the new time $\tau$ (usually called {\em eccentric anomaly\/}) and defined by 
\begin{equation}\label{eq:time-rescaling}
d\tau = \dfrac{1}{r(t)}dt.
\end{equation}
It is worth noticing that the scaling function is a priori unknown (as it is dependent on the unknown function $r$). Thus, we get 
\begin{equation}
\dfrac{d}{dt}= \dfrac{1}{r}\dfrac{d}{d\tau}, \qquad \dfrac{d^2}{dt^2}=\dfrac{1}{r^2}\dfrac{d^2}{d\tau^2} - \dfrac{r'}{r^3}\dfrac{d}{d\tau}
\end{equation}
where we denote $\frac{d}{d\tau}$ by a prime $\cdot'$. The energy level $h$ given in Equation~\eqref{eq:energy-level} in the new time variable can be written as follows
\begin{equation}\label{eq:energy-new-time}
\mu^2 r'^2+ k^2-2m\mu r= 2\mu r^2 h
\end{equation}
and thus the radial Kepler equation given in Equation~\eqref{eq:radial-kepler}, reduces to the following linear second order (non-homogeneous ordinary differential) equation
\begin{equation}\label{eq:kepler-radial-new-time}
	\mu r''-2r h-m=0 \quad  \textrm{ on } \   (0, \mathcal T).
\end{equation}
We observe also that, in the new time variable  $\tau$, the (prime)  period of the solution is  given by 
 \begin{equation}\label{eq:period-frequency}
 \mathcal T= \dfrac{2\pi}{\omega} \quad \textrm{ where } \quad \omega\:= \sqrt{\dfrac{2\abs{h}}{\mu }}.
 \end{equation}
 \begin{rem}
According to the value of the  energy $h$ and the angular momentum $k$, six cases can appear. However in this paper we are interested only in the bounded non-collision  motions which correspond to the case of non-zero angular momentum and negative energy. 
\end{rem}
It is well-known, in fact, that all solutions can be written in terms of the orbital elements and in the particular case of non-zero angular momentum and negative energy, such   solutions are ellipses given by 
 \begin{equation}\label{eq:reta}
 r(\tau)=a[1-e\cos \omega \tau], \quad \textrm{ where }\quad \sqrt{\dfrac{m}{\mu a^3}}\, t=\dfrac{2\pi}{T} t= \tau -e \sin \tau, \quad a \textrm{ is the semi-major axis}
 \end{equation}
and where $\varepsilon\:= [1+ 2hk^2/(\mu m^2)]^{1/2}$ denotes the eccentricity.  We observe that in the circular case $\varepsilon =0$ and $r(\tau)=a$ is indeed constant. It is also well-known that in polar coordinates $(r,\vartheta)$, the polar equation of the ellipses is given by 
\begin{equation}\label{eq:semilatusrectum}
r(\vartheta)=\dfrac{r_0}{1-e \cos \vartheta}\quad  \textrm{ where }\quad  r_0\:=\dfrac{k^2}{\mu m}.
\end{equation}
$r_0$ is called {\em semi-latus rectum\/} and it is related to the eccentricity and the semi-major axis of the ellipses by $a= r_0/(1-e^2)$.

\section{Maslov index and Conley-Zehnder index}\label{sec:Preliminaries-on-Maslov-and-Conley-Zehnder-indexes}

This section is devoted to some classical definitions and basic properties of the Conley-Zehnder index.  More precisely we need its generalization to paths with degenerate endpoints  in the $2$ and $4$ dimensional symplectic space as well as a description of the Maslov index as an intersection index in the Lagrangian Grassmannian setting.


\subsection{On the  Conley-Zehnder index}\label{subsec:CZ}

We consider the standard symplectic space $(\R^{2n}, \omega_0)$ where
 the {\em (standard) symplectic form\/} $\omega$ is defined by $ \omega(\cdot, \cdot)=\langle J \cdot, \cdot\rangle$ and 
 where $J$ is the $2n \times 2n$ matrix  given  by 
$J\:=\begin{bmatrix} 	
0 & -\Id\\ \Id & 0
\end{bmatrix}$.
We denote by $\Sp(2n)$ the \emph{ symplectic group} defined by $\Sp(2n)  \:= \Set{M \in \GL(2n) | \trasp{M}JM = J}$ and for $ 0 \leq k \leq 2n$ we set $\Sp_k(2n) \:= \Set{M \in \Sp(2n)|\dim\ker(M - I)=k}$. 	Thus, in particular, we get that 
			\[
		\Sp(2n) = \bigcup_{k = 0}^{2n} \Sp_k(2n).
		\]
Given $M \in \Sp(2n)$, we set $I(M)\:=(-1)^{n-1}\det (M- \Id)$ and we define the following hypersurface of 
	\[
	\Sp^0(2n)= \Set{M \in \Sp(2n)| I(M)=0} \subset \Sp(2n).
	\]
Now, setting $\Sp(2n)^*\:= \Sp(2n)\setminus \Sp^0(2n)$ and denoting by  
	\[
	\Sp^+(2n)\:= \Set{M \in \Sp(2n)| I(M)<0}  \quad \textrm{ and } \quad
	\Sp^-(2n)\:= \Set{M \in \Sp(2n)| I(M)>0} 
	\]
it is possible to prove that these two are the only two path connected components of $\Sp(2n)^*$ and are simply connected in $\Sp(2n)$.
	
For any $M \in \Sp^0(2n)$, we can define a transverse orientation at $M$  through the positive direction $\frac{d}{dt}Me^{tJ}|_{t=0}$ of the path $t \mapsto Me^{tJ}$ with $t \geq 0$ sufficiently small. We define the following set 
	\begin{equation}\label{eq:pt}
		\mathscr P_{T}(2n) \:= \Set{ \psi \in \mathscr{C}^0 \bigl( [0, T]; \Sp(2n) \bigr) | \psi(0) = \Id}.
	\end{equation}
In the following will be essentially interested to the cases $\mathscr P_{T}(2)$ and $\mathscr P_{T}(4)$.


\subsubsection*{The Conley-Zehnder index in $\Sp(4)$}

Since later on we shall work in  $\Sp(4)$, in this section we restrict to $\Sp(4)$ and $\Sp(2)$ in order to simplify the presentation.

Consider now the two square matrices $2 \times 2$ matrices $M_1$ and $M_2$ given by 			\[
				M_k \:= \begin{pmatrix}
						a_k & b_k \\
						c_k & d_k
					\end{pmatrix} \in \Sp(2)\quad \textrm{ for }\quad  k= 1, 2,
			\]
		where $a_k, b_k,c_k, d_k \in \R$. The \emph{symplectic  sum} of $M_1$ and $M_2$ is defined  as the following
		$4 \times 4$ matrix below
			\begin{equation} \label{def:diamondproduct}
				M_1 \diamond M_2 \:= \begin{pmatrix}
									a_1 & 0 & b_1 & 0 \\
									0 & a_2 & 0 & b_2 \\
									c_1 & 0 & d_1 & 0 \\
									0 & c_2 & 0 & d_2
								\end{pmatrix}\in \Sp(4).
			\end{equation}
		The $2$-fold symplectic sum of $M$ with itself is denoted by $M^{\diamond 2}$. The \emph{symplectic sum} of two paths $\psi_j \in \mathscr P_{T}(2)$, with $j=1,2$, is defined by:
					\[
				(\psi_1 \diamond \psi_2)(t) \:= \psi_1(t) \diamond \psi_2(t), \quad \forall\, t \in [0, T].
			\]
Given any two continuous paths $\phi_1, \phi_2 : [0, T] \to \Sp(4)$ such that $\phi_1(T) = \phi_2(0)$,
we denote by $*$ their \emph{concatenation}. We also define a special continuous symplectic path $\xi_2 : [0, T] \to \Sp(4)$ as follows:
	\begin{equation} \label{eq:xin}
		\xi_2(t) \:= \begin{bmatrix}
						2 - \dfrac{t}{T} & 0 \\
						0 & \biggl( 2 - \dfrac{t}{T} \biggr)^{-1}
					\end{bmatrix}^{\diamond 2}= \begin{bmatrix}
						\left(2 - \dfrac{t}{T}\right) \Id & 0  \\
						0  & \biggl( 2 - \dfrac{t}{T} \biggr)^{-1}\Id
					\end{bmatrix}, \quad \forall\, t \in [0, T]
		\end{equation}
		where in the righthand side $\Id$ denotes the identity $2 \times 2$. Setting $D(a)\:=\diag[a,a^{-1}]$, we define the two matrices $M_2^+= D(2)\diamond D(2) \in \Sp^+(4)$ and $M_2^-= D(-2) \diamond D(2)\in \Sp^-(4)$.  
		
		Given  $\psi \in \P_T(2)$ and $m \in \N \setminus \{0\}$, the \emph{$m$-th iteration of $\psi$} is $\psi^m : [0, m T] \to \Sp(2)$ defined as
		\[
			\psi^m(t) \:= \psi(t - jT) \bigl(\psi(T) \bigr)^j, \qquad \text{for } j T \leq t \leq (j + 1)T, \quad  j = 0, \dotsc, m - 1.
		\]
\begin{rem}
All definitions given in this subsection in dimension $2$ or $4$ can be carried over in any even dimension.
\end{rem}
We now introduce   the definition of $\iCZ$-index which is an intersection index between a symplectic path starting from identity and the singular hypersurface defined above.
\begin{defn}\label{def:indicediMaslov}
		We consider a continuous path $\psi:[0,T] \to \Sp(4)$ such that $\psi(0)=\Id$. We define the {\em Conley-Zehnder index\/} of the path $\psi$ as the integer given by 
					\begin{equation} \label{eq:indiceomega}
				\iCZ(\psi(t), t \in [0, T]) \:= \bigl[ e^{-\varepsilon J} \psi * \xi_2 : \Sp^0(4)\bigr]
			\end{equation}
			where the righthand side in Equation~\eqref{eq:indiceomega} 
		 is the usual homotopy intersection number and $\varepsilon$ is a positive real sufficiently small  number . 
	\end{defn}
	\begin{rem}
	It is worth to observe that the advantage to perturb the path $t \mapsto (\psi *\xi_2)(t) $ is in order to get a new path having  nondegenerate endpoints (and in particular of nondegenerate endpoint which is the original assumption made for defining the Conley-Zehnder index). We also notice that  
	the advantage to concatenate the original path $\psi$ with $\xi_2$ is in order to simplify some relations in the specific case in which the symplectic path $\psi$ is the fundamental solution of a Hamiltonian system whose Hamiltonian is the Fenchel transform of a Lagrangian one. 
		\end{rem}


	\subsubsection*{Properties of the $\iCZ$-index}
		We list some of the  the basic properties of the index that we need to use later on.
	\begin{enumerate}[(i)]
		\item {\em ($\diamond$-additivity)\/} Let $\psi_1: [0, T] \to \Sp(2)$ and $\psi_2:[0,T]\to \Sp(2)$ be two symplectic paths. Then we have
				\[
					\iCZ\big((\psi_1 \diamond \psi_2)(t), t \in [0, T]\big)= \iCZ\big(\psi_1(t), t \in [0, T] ) + \iCZ(\psi_2(t), t \in [0, T]).
				\]
		\item {\em (Homotopy invariance)\/} For any two paths $\psi_1$ and $\psi_2$, if $\psi_1 \sim \psi_2$ in $\Sp(2)$ with either fixed or always non-degenerate endpoints,
			there holds
				\[
					\iCZ(\psi_1(t), t \in [0, T])=\iCZ(\psi_2(t), t \in [0, T]).
				\]
		\item {\em(Affine scale invariance)\/} For all $k>0$ and $\psi \in \mathscr{P}_{kT}(4)$, we have
				\[
					\iCZ\big(\psi(kt), t \in [0, T]\big)= \iCZ\big(\psi(t), t \in [0,kT]\big).
				\]
	\end{enumerate}

\subsubsection*{The geometric structure of $\Sp(2)$}

	The symplectic group $\Sp(2)$ captured the attention of I.~Gelfand and V.~Lidskii first, who in 1958 described a toric representation of it. 	The $\R^3$-cylindrical coordinate representation of $\Sp(2)$, that we shall use throughout,  was introduced by Y.~Long in 1991. 	Every real invertible matrix $A$ can be decomposed in \emph{polar form}
		\[
			A = PO,
		\]
	where $P \:= (A\trasp{A})^{1/2}$ is symmetric and positive definite and $O$ is a proper rotation:
		\[
			O = \begin{pmatrix}
					\cos \theta & -\sin \theta \\
					\sin \theta & \cos \theta
				\end{pmatrix}.
		\]
In particular, the matrix $P$ can be written in the following form (cfr. \cite[Appendix A ]{BJP16} and references therein for further details )
		\[
			P = \begin{pmatrix}
					r & z \\
					z & \frac{1 + z^2}{r}
				\end{pmatrix}
		\]
	and then every symplectic $2\times 2$ matrix $M$  can be written as the product
		\begin{equation} \label{eq:sympldec}
			M = \begin{pmatrix}
					r & z \\
					z & \frac{1 + z^2}{r}
				\end{pmatrix}
				\begin{pmatrix}
					\cos \theta & -\sin \theta \\
					\sin \theta & \cos \theta
				\end{pmatrix},
		\end{equation}
	where $(r, \theta, z) \in (0, +\infty) \times [0, 2\pi) \times \R$. Viewing $(r, \theta, z)$ as cylindrical coordinates in $\R^3 \setminus \{ \text{$z$-axis} \}$ we obtain a smooth global diffeomorphism $\Psi : \Sp(2) \to \R^3 \setminus \{ \text{$z$-axis} \}$. We shall henceforth identify elements in $\Sp(2)$ with their 	image under~$\Psi$.
	
	\begin{figure}[tb]
	\begin{center}
		\includegraphics{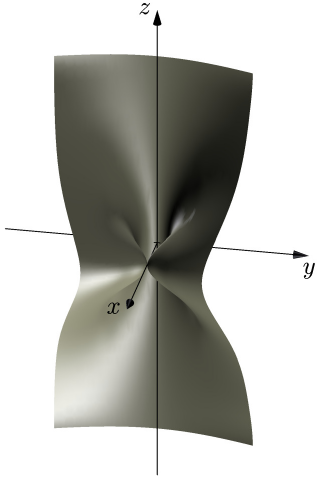}
		\caption{The singular surface $\Sp(2)_1^0$. The representation is in cylindrical coordinates $(x, y, z) = (r \cos\theta, r\sin\theta, z)$.} \label{fig:Sp(2)10}
	\end{center}
	\end{figure}
		
	The eigenvalues of a symplectic matrix $M$ written as in Equation~\eqref{eq:sympldec} are
		\[
			\lambda_\pm \:= \frac{1}{2r} \Bigl[ (1 + r^2 + z^2) \cos \theta \pm \sqrt{(1 + r^2 + z^2)^2 \cos^2 \theta - 4r^2} \Bigr].
		\]
		Thus, we get
		\[
			\begin{split}
				I(M) & =  2 - \left( r + \frac{1 + z^2}{r} \right) \cos \theta
			\end{split}
		\]
	and define
		\begin{equation*}\begin{split}
			\Sp^\pm(2) &= \Set{ (r, \theta, z) \in (0, +\infty) \times [0, 2\pi) \times \R | \pm(1 + r^2 + z^2) \cos \theta > 2r }, \\
			\Sp^0(2)  &= \Set{ (r, \theta, z) \in (0, +\infty) \times [0, 2\pi) \times \R | \pm(1 + r^2 + z^2) \cos \theta = 2r  }.
		\end{split}\end{equation*}
	The set $\Sp^*(2)\:=\Sp^+(2) \cup \Sp^-(2)$ is named the \emph{regular part} of $\Sp(2)$, while $\Sp^0(2)$ is its \emph{singular part}; the former
	corresponds to the subset of $2 \times 2$ symplectic matrices which do not have $1$ as an eigenvalue, whereas those matrices admitting $1$ in their spectrum belong to the
	latter.
	We are particularly interested in $\Sp^0(2)$, the singular part of $\Sp(2)$ associated with the eigenvalue $1$, 	a representation of which is depicted in Figure~\ref{fig:Sp(2)10}. The ``pinched'' point is the identity matrix, and it is the only element satisfying $\dim\ker(M - I) = 2$.
	If we denote by
		\[
			\Sp(2)_{\pm}^0 \:= \Set{ (r, \theta, z) \in \Sp^0(2)| \pm \sin \theta > 0 },
		\]
	we see that $\Sp^0(2) \setminus \{ \Id \} = \Sp(2)_{+}^0 \cup \Sp(2)_{-}^0$, and each subset is a path-connected component diffeomorphic to $\R^2 \setminus \{ 0\}$.
\begin{figure}[hh]
\centering
\includegraphics{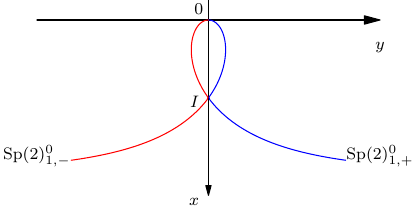}
\caption{Intersection of $\Sp^0(2)$ with the plane $z = 0$. The representation is in Cartesian coordinates $(x, y) = (r \cos\theta, r\sin\theta)$.} \label{fig:curva}
	\end{figure}
	In Figure~\ref{fig:curva},  it is represented as a horizontal section (specifically the intersection with the plane $z=0$) of the hypersurface $\Sp^0(2)$. 
	
	What the $\iCZ$ index of a symplectic path in $\Sp(2)$ starting from the identity actually counts are the algebraic (signed) intersections of the path with the surface $\Sp^0(2)$. If we imagine projecting the symplectic path into the horizontal plane $\{ z = 0 \}$ a  simple way to think about this index is an algebraic count of  intersection  with the  curve  depicted in Figure~\ref{fig:curva}of a continuous path (which is  the projection of the original one)  on the $\widehat{x O y}$-plane. (For further details, we refer the interested reader to \cite{BJP16} and references therein). We observe that, in these coordinate system  the (unbounded) path connected component $\Sp^+(2)$ depicted in Figure~\ref{fig:curva}, corresponds to the region of the plane containing the portion of the  $x$-axis emanating from the identity.


\subsection{The Maslov index in $\Lambda(2)$}\label{subsec:Maslov}
In the $4$-dimensional standard symplectic space $(\R^4,\omega_0)$ a {\em Lagrangian subspace\/} is an $2$-dimensional subspace $L \subset \R^4$ on which $\omega_0$ vanishes identically. It is well-known \cite{Dui76} that the set of all Lagrangian subspaces $\Lambda(2)$, usually called the {\em Lagrangian Grassmannian of the $2$-dimensional  Lagrangian subspaces\/} has the structure of a three dimensional  compact, real-analytic submanifold of the  Grassmannian of all $2$-dimensional subspaces of $\R^4$.
 The real-analytic atlas can be described as follows. For all $L \in \Lambda(2)$ and $k \in {0, 1}$, let $\Lambda_0(L)$ denote the dense and open set of all Lagrangian subspaces $L'$ such that $L'\cap L_0=\{0\}$. Given $L_1\in \Lambda(2)$, let  $L_0 \in \Lambda(L_1)$ and let us consider the map 
\[
\varphi_{L_0, L_1}: \Lambda_0(L_1) \to \Bsym(L_0) \textrm{ defined by } \varphi_{L_0, L_1}(L)=\left.\omega_0(\cdot, T \cdot)\right|_{L_0 \times L_0}
\]
where $T:L_0 \to L_1$ denotes the unique linear map whose graph is the  Lagrangian (subspace) $L$. \footnote{%
The symmetry of the restriction of the bilinear map $\omega_0(T\cdot, \cdot)$ onto $L_0 \times L_0$ is consequence of the fact that $L$ is Lagrangian.}

Let $L_0 \in \Lambda(2)$ and for $j \in \{0,1,2\}$, we set 
\begin{equation}
\Lambda_j(L_0)=\Set{L \in \Lambda(2)| \dim(L\cap L_0)=j}.
\end{equation}
It is easy to see that $\Lambda_j(L_0)$ is a connected $j(j+1)/2$-codimensional submanifold of $\Lambda(2)$ and in particular $\Lambda_2(L_0)=L_0$. Moreover, the set 
\begin{equation}
	\Sigma(L_0)= \Lambda_1(L_0) \cup\Lambda_2(L_0)
\end{equation}
is the (topological) closure of the top stratum $\Lambda_1(L_0)$ usually  called the {\em Maslov cycle\/}. $\Lambda_1(L_0)$ has a canonical transverse orientation, meaning that there exists $\delta>0$ for each $L \in \Lambda_1(L_0)$, the path of Lagrangian subspaces $t\mapsto e^{tJ}L$ for $t \in (-\delta, \delta)$ crosses transversally $\Lambda_1(L_0)$ and as $t$ increases the path is pointing towards the transverse positive direction. Thus this cycle is two-sidedly embedded in $\Lambda(2)$.
Following authors in \cite{CLM03} we introduce the following definition.
\begin{defn}\label{def:L-Maslov}
Let $L_0 \in \Lambda(2)$ and let $\ell:[0, T] \to \Lambda(2)$ be a continuous path. We term {\em $\iCLM$-index\/} the integer defined by 
\begin{equation}\label{eq:iclm}
	\iCLM(L_0, \ell(t), t \in [0, T])\:= \big[e^{-\epsilon J}\, \ell(t): \Sigma(L_0)\big]
\end{equation} 
where the right-hand side is the intersection number and $\epsilon \in (0,1)$ is sufficiently small. 
\end{defn} 
\begin{rem}\label{rem:utile-dopo}
We observe that the $\iCLM$-index given in Definition \ref{def:L-Maslov} could be defined by using the Seifert-Van Kampen theorem for groupoids. More precisely, we denote throughout by $\pi(\Lambda(2))$ the {\em fundamental groupoid\/} of $\Lambda(2)$, namely the set of fixed-endpoints homotopy classes $[\ell]$ of continuous paths $\gamma$ in $\Lambda(2)$, endowed with the partial operation of concatenation $*$. For all $L_0 \in \Lambda(2)$, there exists a unique $\Z/2$-valued groupoid homomorphism $\iRS: \pi(\Lambda(2)) \to  \Z/2$ such that 
\begin{equation}\label{eq:local-definition-L-maslov}
	\iRS(\ell(t),L_0,  t \in [0,1])= \dfrac12\sgn\varphi_{L_0,L_1}\big(\ell(1)\big)-\dfrac12\sgn \varphi_{L_0,L_1}\big(\ell(0)\big)
\end{equation}
for all continuous curve $\ell:[0,1] \to \Lambda_0(L_1)$ and for all $L_1 \in \Lambda_0(L_0)$.\footnote{
This index was defined in a slightly different manner by authors in \cite{RS93}. 
} By \cite[Equation (3.7)]{LZ00} we get that 
\[
\iCLM(L_0, \ell(t), t \in [0,1])= \iRS(\ell(t), L_0, t \in [0,1]) + \dfrac12\big[h(0)-h(1)\big]
\]
where $h(i)=\dim\big(L_0 \cap \ell(i)\big)$ for $i=0,1$. Thus locally the $\iCLM$-index with respect to the fixed Lagrangian $L_0$ could be defined equally well as the unique $\Z$-valued groupoid homomorphism $\iRS: \pi(\Lambda(2)) \to  \Z/2$ such that 
\begin{equation}\label{eq:local-definition-L-maslov-2}
	\iRS(\ell(t),L_0,  t \in [0,1])= \Coindex\varphi_{L_0,L_1}\big(\ell(1)\big)-\Coindex\varphi_{L_0,L_1}\big(\ell(0)\big).
\end{equation}
A different choice has been considered by authors in \cite{GPP04}. 
\end{rem}
\begin{rem}
It is well-known that the Lagrangian Grassmannian can be realized as homogeneous space, through an action of the unitary group. In fact it is easy to show that 
\[
\Lambda(2)= \UU(2)/\O(2).
\]	
As proved by Arnol'd in \cite[Section 3]{Arn86}, $\Lambda(2)$ is the  nonoriented total space of the nontrivial bundle with fiber $S^2$ and with base the circle. The proof provided by Arnol'd is based on the short exact sequences of fibrations of the unitary and orthogonal Lie groups. 
\end{rem}


\subsubsection*{Computing $\iCLM$-index through crossing forms}

The $\iCLM$-index  defined above, is in general, quite hard to compute. However one efficient way to do so is via crossing forms as  introduced by authors in \cite{RS93}. 
	 Let $\ell:[0,T] \to \Lambda(2)$ be a $\mathscr C^1$-curve of Lagrangian subspaces and let $L_0 \in \Lambda(2)$. Fix $t \in [0, T]$ and let $W$ be a fixed Lagrangian complement of $\ell(t)$.
	If $s$ belongs to a suitable small neighborhood of $t$ for every $v \in l(t)$ we can find a unique vector $w(s)\in W$ in such a way that $v + w(s) \in l(s)$.
	\begin{defn}\label{def:crossig-form}
		The \emph{crossing form} $\Gamma(\ell, L_0, t_*)$ at $t_*$ is the quadratic form $\Gamma(\ell, L_0, t_*) : l(t^{*})\cap L_0 \to \R$ defined by
			\begin{equation}\label{eq:crossingform}
				\Gamma(l, L_0, t_*)[v] \:= \dfrac{d}{ds}\omega_0\big(v,w(s)\big)\big\vert_{s=t_*}.
			\end{equation}
		The number $t_*$ is said to be a \emph{crossing instant for $l$ with respect to $L_0$} if  $l(t_*) \cap L_0 \neq \{0\}$ and it is called \emph{regular} if the crossing form is non-degenerate.
	\end{defn}  
	Let us remark that regular crossings are isolated and hence on a compact interval there are finitely many. The following result is well-known.
	\begin{prop} (\cite[ Theorem~3.1]{LZ00})
	Let $L_0 \in \Lambda(2)$ and $\ell \in \mathscr C^1\big([0,T], \Lambda(2)\big)$ having  only regular crossings.  Then the $\iCLM$-index of $\ell$ with respect to $L_0$ is given by 
		 	\begin{equation}\label{eq:CLMmaslovestremideg}
		 		\iCLM(L_0, \ell, [0, T]) \:= \Coindex\Gamma(\ell, L_0, 0) + \sum_{t_* \in (0,T)} \sgn \Gamma(\ell,L_0,t_*) - \Index \Gamma(\ell,L_0,T) 
		 	\end{equation}
		 where the summation runs over all crossings instants $t_* \in (0, T)$.
	\end{prop}
On the Euclidean space  $V \:= \R^2 \times \R^2$, we introduce the  symplectic form $\omega_{\mathcal{J}}: V \times V \to \R$ defined  by 
		\[
			\omega_{\mathcal{J}} \big(v, w\big)\:= \langle \mathcal J v, w\rangle, \qquad \forall\, v, w \in V \quad 
\textrm{ where }\quad 
			\mathcal{J} \:= \begin{pmatrix}
						-J &0\\
						0 & J
					\end{pmatrix}.
		\]
	By a direct calculation it follows that, if $M \in \Sp(2)$ then its graph $\Gr(M) \:= \{\trasp{(x, Mx)}|x \in \R^2\}$ 
	is a Lagrangian subspace of the  symplectic space $(V,\omega_{\mathcal{J}})$.	Thus, a  path of symplectic matrices $t\mapsto \psi(t)$ induces a path of Lagrangian subspaces  of $(V, \omega_{\mathcal{J}})$ defined through its graph  by 
	$t \mapsto \Gr(\psi(t))$.  
	
		The next result  (in the general setting) was proved by authors in \cite[Corollary 2.1]{LZ00} (cfr. \cite[Lemma 4.6]{HS09}) and in particular put on evidence the relation between the  $\iCZ$-index  associated to a path of symplectic matrices and  the $\iCLM$-index of the corresponding path of Lagrangian subspaces with respect to the diagonal $\Delta \:= \Gr(\Id)$.
		\begin{prop} \label{thm:chiave}
		For any continuous symplectic path $\psi \in \mathscr{P}_{T}(2n)$ (thus starting at the identity), we get the following  equality 
		\[
		\iCZ(\psi(t), t \in [0, T]) + n= \iCLM\big(\Delta, \Gr(\psi(t)), t \in [0, T]\big)
		\]	
		In the special case $n=1$,  if  $\phi \in \mathscr{P}_{T}(2)$, it holds that 
		\[
		\iCZ(\phi(t), t \in [0, T]) + 1= \iCLM\big(\Delta, \Gr(\phi(t)), t \in [0, T]\big).
		\] 
		\end{prop} 
	We conclude this section with an index theorem which relates the Morse index $\iMor(\gamma)$ of a periodic solution $\gamma$ of Equation \eqref{eq:kepler-one-center} seen as critical point of the Lagrangian action functional  and the $\iCZ$-index of the fundamental solution	$\psi$	of the linearized Hamiltonian system at $\gamma$. 
		\begin{prop}[Morse Index Theorem, {\cite[page~172]{Lon02}}] \label{thm:indextheorem}
		Under the above notation, we have 
			\[
				\iMor(\gamma) 
				= \iCZ(\psi(t), t \in [0,T]).
			\]
	\end{prop}


\section{Some explicit computations for paths in $\Sp(2)$}\label{subsec:CZ-2}

The aim of this subsection is to explicitly compute the Conley-Zehnder index as well as the $\iCLM$-index introduced in Subsection~\ref{subsec:reduction} using crossing forms.

\begin{lem}\label{thm:crossing-form}
Let $\phi:[0, T] \to \Sp(2)$ be the path 
\[
	 		\phi(t) \:= \begin{pmatrix}
	 					a(t) & b(t)\\ 
	 					c(t) & d(t)
	 				   \end{pmatrix},\qquad t \in [0, T]
\]
with $a,b,c,d \in \mathscr{C}^1([0, T],\R)$ and  let  $\ell(t) \:= \Gr\big(\phi(t)\big)\in \Lambda(2)$.
	We  assume that $t_*\in [0, T]$ is a crossing instant for $\ell$ (with respect to $\Delta$) such that $\ell(t_*)\cap \Delta\neq \{0\}$. Assuming that $d(t_*)\neq 0$, then   the  crossing form at  $t=t_*$ is given by
		\begin{equation}\begin{split}\label{eq:crossingrot1}
			&\Gamma(l,\Delta,t_*)(v) = -x_0 \eta'(t_*) - y_0 \xi'(t_*)\quad \textrm{ where } \\ 
			&\xi'(t_*) = a'(t_*) x_0 + b'(t_*)y_0 - \frac{b(t_*)}{d(t_*)}\big[c'(t_*)x_0+d'(t_*)y_0\big],\\
	 		&\eta'(t_*) = - \frac{1}{d(t_*)}\big[c'(t_*)x_0+d'(t_*)y_0\big],\qquad \forall\, v \:= (x_0,y_0,x_0,y_0) \in \Delta.
		\end{split}\end{equation}
\end{lem}
\begin{proof}
	In order to compute the crossing form given in Equation~\eqref{eq:crossingform}, we first consider the Lagrangian subspace 
		\[
			W \:= \{0\}\times \R \times \R \times \{0\}
		\] 
	and we observe that this gives a Lagrangian decomposition of $\R^4$, specifically $\R^4 = \Delta \oplus W$. Now,  for any $v \:= (x_0,y_0,x_0,y_0) \in \Delta$ let us choose $w(t) \:= (0,\eta(t),\xi(t),0)\in W$ in order that $v + w(t) \in l(t)$. This
	means that $\eta(t)$ and $\xi(t)$ solve the equations
		\begin{equation}
 			x_0 + \xi(t) = a(t) x_0 + b(t)\big(y_0 + \eta(t)\big),
			\qquad
 			y_0 = c(t) x_0 + d(t)\big(y_0 + \eta(t)\big).
		\end{equation}
	Since in a crossing instant $t_*$ we have $\xi(t_*)=\eta(t_*)=0$, differentiating the above identities gives
	 		\begin{equation}\label{eq:lederivate}
			\begin{split}
	 			\xi'(t_*) &= a'(t_*) x_0 + b'(t_*)y_0 - \frac{b(t_*)}{d(t_*)}\big[c'(t_*)x_0+d'(t_*)y_0\big],\\
	 			\eta'(t_*) &= - \frac{1}{d(t_*)}\big[c'(t_*)x_0+d'(t_*)y_0\big].
			\end{split}
	 		\end{equation}			 	
	By a direct computation we obtain
		\begin{equation}
	 		\omega_{\mathcal J}\big(v,w(t)\big) = \langle \mathcal J v, w(t) \rangle
							 = -\left\langle J \begin{pmatrix}
									x_0\\ y_0
								\end{pmatrix},
								\begin{pmatrix}
									0\\ \tau(t)
								\end{pmatrix}\right\rangle
								+\left\langle J \begin{pmatrix}
									x_0\\y_0
								\end{pmatrix},  \begin{pmatrix}
	 								\xi(t)\\0
								\end{pmatrix}\right\rangle\\
							 = -x_0 \eta(t) - y_0 \xi(t).
		\end{equation}
	Hence the crossing form at the crossing instant $t=t_*$ is given by
		\begin{equation}\label{eq:crossingrot2}
			\Gamma(\ell,\Delta,t_*)(v) = \dfrac{d}{dt} \omega_{\mathcal J}\big(v,w(t)\big)\Big\vert_{t=t_*} = -x_0 \eta'(t_*) - y_0 \xi'(t_*).
		\end{equation}
\end{proof}
\begin{rem}
If $d(t_*)=0$, then it is enough to replace the path $t\mapsto \ell(t)$ by  the path $t\mapsto\ell_\eps(t)\:=\Gr(e^{-\eps Jt} \phi(t))$ with $\eps>0$ sufficiently small. By the well-definedness  of the $\iCLM$-index (cfr. \cite[pag. 138]{CLM03}) the result follows. 
\end{rem}
We are going to apply the computation provided in Lemma~\ref{thm:crossing-form} in some specific cases that we shall need later.  
\begin{ex} \label{ex:Romega}
		We let   $a_1, a_2 \in \R^+$, $\beta\:= \sqrt{a_1a_2}$ and we  consider the path $R_\beta: [0, T] \to \Sp(2)$ defined by
	 		\[
	 		R_\beta(t)=
	 		\begin{bmatrix}
\cos(\beta\, t) & -\dfrac{\beta}{a_1} \sin (\beta\, t)\\
\dfrac{\beta}{a_2}\sin(\beta\, t) &  \cos (\beta\, t)
\end{bmatrix}.
	 		\]
	 		We aim to compute the crossing form with respect to $\Delta$ of the Lagrangian path $t\mapsto\ell_\beta(t)=\Graph\big(R_\beta(t)\big)$.	 Bearing in mind previous notation, we get 
	 		\[
	 		 a(t) = d(t) = \cos(\beta\, t),  \quad 
	 		 b(t) =  - \dfrac{\beta}{a_1} \sin (\beta\, t), \qquad 
	 		 c(t) = \dfrac{\beta}{a_2}\sin(\beta\, t).
	 		\]
	 		 We  observe that $t_*$ is a crossing instant if and only if 
	 	$t_*\in \dfrac{2\pi}{\beta} \Z$. Moreover by a direct calculation we get 
	 	\begin{equation}
		\begin{split}
	 	a(t_*)&=d(t_*)= 1,\qquad b(t_*)=c(t_*)= 0. \quad \intertext{ Moreover, }
	 	a'(t_*)&=-\beta \sin(\beta t_*)= d'(t_*), \qquad b'(t_*)=-a_2\cos(\beta t_*), \qquad c'(t_*)=a_1\cos(\beta t_*)\\
	 	a'(t_*)&=0= d'(t_*) \qquad b'(t_*)=-a_2,\qquad c'(t_*)=a_1.
		\end{split}
	 	\end{equation}
	 	Thus $\xi'(t_*)=-a_2y_0$ and 
	 	$\eta'(t_*)=-a_1x_0$.  
		Using Lemma~\eqref{thm:crossing-form},  we directly get
	 		\begin{equation}\label{eq:crossingrotfin1}
	 			\Gamma(\ell_\beta,\Delta,t_*)[v] = -x_0 \eta'(t_*)-y_0 \xi'(t_*) = a_2 x_0^2+ a_1 y_0^2.
		 	\end{equation}
	 	 Since $\Gamma(\ell_\beta,\Delta,t_*)$ is a positive definite quadratic form on a $2$ dimensional vector space, it follows that it is non-degenerate and its signature (which coincides with the coindex) is $2$.
	 	Summing up all these computations we obtain
	 	\begin{equation}\label{eq:CLM_omega}
	 		 \iCLM(\Delta, \ell_\beta(t), t \in [0, T]) =
	 		 \begin{cases}
	 		 2\left\lfloor\dfrac{ T \beta}{2\pi}\right\rfloor & \textrm{ if } T \in \dfrac{2\pi}{\beta}\Z \\
	 		 \\
	 		2\left\{\left\lfloor\dfrac{T\beta}{2\pi}\right\rfloor+1\right\} & \textrm{ if }T \notin \dfrac{2\pi}{\beta}\Z,
	 		 \end{cases}	
	 	\end{equation}
	 	 where $\lfloor \cdot \rfloor$ denotes the greatest integer less than or equal to its argument.
	 	\end{ex}
\begin{rem}
Let us consider  the path $S_\beta: [0, T] \to \Sp(2)$ defined by
	 		\[
	 		S_\beta(t)=
	 		\begin{bmatrix}
\cos(\beta\, t) & \dfrac{\beta}{a_1} \sin (\beta\, t)\\
-\dfrac{\beta}{a_2}\sin(\beta\, t) &  \cos (\beta\, t)
\end{bmatrix}
	 		\]
and let us define the Lagrangian path $t\mapsto m_\beta(t)=\Gr\big(S_\beta(t)\big)$. By the very same calculations as before, we get  
 \begin{equation}\label{eq:CLM_omegaS}
	 		 \iCLM(\Delta, m_\beta(t), t \in [0, T]) =
	 		 \begin{cases}
	 		 -2\left\lfloor\dfrac{ T \beta}{2\pi}\right\rfloor & \textrm{ if } T \notin \dfrac{2\pi}{\beta}\Z \\
	 		 \\
	 		-2\left\{\left\lfloor\dfrac{T\beta}{2\pi}\right\rfloor+1\right\} & \textrm{ if }T \in \dfrac{2\pi}{\beta}\Z,
	 		 \end{cases}	
	 	\end{equation}
 
 \end{rem}
Summing up we have the following result. 
\begin{lem}\label{thm:Gutt41}
Let us consider the path 
\[
\psi: [0, T] \longrightarrow \Sp(2): t \longmapsto e^{t J S}
\]	
where $S$ is either symmetric positive or negative definite and for every $t \in [0, T]$, we let $n(t)\:=\Graph\big(\psi(t)\big)$. Thus, we get  
\begin{equation}
\iCLM (\Delta, n(t),t \in [0, T])= 
\begin{cases}
2\left\lfloor \dfrac{\sqrt{a_1a_2}\,T}{2\pi}\right\rfloor \qquad \textrm{ if S is positive definite and }  T\in \dfrac{2\pi}{\sqrt{a_1a_2}}\Z\\
\\
-2\left\lfloor \dfrac{\sqrt{a_1a_2}\,T}{2\pi}\right\rfloor \qquad \textrm{ if S is negative definite and }  T\notin \dfrac{2\pi}{\sqrt{a_1a_2}}\Z\\
\\
2\left\{\left\lfloor \dfrac{\sqrt{a_1a_2}\,T}{2\pi}\right\rfloor +1\right\}\quad \textrm{ if S is positive definite and }  T \notin \dfrac{2\pi}{\sqrt{a_1a_2}}\Z\\
\\
-2\left\{\left\lfloor \dfrac{\sqrt{a_1a_2}\,T}{2\pi}\right\rfloor +1\right\}\quad \textrm{ if S is negative definite and }  T \in \dfrac{2\pi}{\sqrt{a_1a_2}}\Z
\end{cases}
\end{equation}
where $a_1$ and $a_2$ are the eigenvalues of $S$.
 \end{lem}
\begin{proof}
Since $S$ is symmetric, we diagonalize it in the orthogonal group and we get a symplectic basis in $\R^2$ of the folloiwing matrices: 
\begin{equation}
S=
\begin{bmatrix}
	a_1&0\\0&a_2
\end{bmatrix}\qquad \textrm{ and } JS=
\begin{bmatrix}
	0&-a_2\\a_1 & 0
\end{bmatrix}.
\end{equation}
We start by setting  
\[
\beta\:= 
\sqrt{a_1 a_2} \qquad \textrm{ if } \quad a_1a_2 >0
\]
 and we observe that by the definition of the matrix exponential we have the following expression for $\psi$:
\begin{enumerate}
\item If the $\sign S >0$ ($a_1>0$ and $a_2>0$), then we get 
\[
\psi(t)=
\begin{bmatrix}
\cos(\beta\, t) & - \dfrac{\beta}{a_1} \sin (\beta\, t)\\
\dfrac{\beta}{a_2}\sin(\beta\, t) &  \cos (\beta\, t)
\end{bmatrix}
\]
\item If the $\sign S <0$
\[
\psi(t)=
\begin{bmatrix}
\cos(\beta\, t) &  \dfrac{\beta}{a_1} \sin (\beta\, t)\\
-\dfrac{\beta}{a_2}\sin(\beta\, t) &  \cos (\beta\, t)
\end{bmatrix}
\]
\end{enumerate}
The proof now follows by invoking Equation~\eqref{eq:CLMmaslovestremideg}, Lemma~\ref{thm:crossing-form} and  Example~\ref{ex:Romega}.
This concludes the proof.
\end{proof}
As a  consequence of Lemma~\ref{thm:Gutt41} and the homotopy invariance of the $\iCLM$-index we get the following result.\begin{prop}\label{thm:me-serve}
Let us consider the path pointwise defined by $\psi_\eps(t)=e^{-\eps J} \psi(t)$, where $\psi$ is given in Lemma~\ref{thm:Gutt41} and  let $n_\eps(t)\:=\Graph\big(\psi_\eps(t)\big)$. Then, for $\eps>0$ sufficiently small,  we get  
\begin{equation}
\iCLM (\Delta, n_\eps(t),t \in [0, T])= 
\iCLM (\Delta, n(t),t \in [0, T]).
\end{equation}
\end{prop}
\begin{proof}
This result is a direct consequence of the well definedness of the $\iCLM$-index. (Cfr.\cite[pag. 138]{CLM03}). 
\end{proof}
Another situation which occurs often in the applications when  there is the presence of a conservation law,  it is described in the following result. 
\begin{lem}\label{thm:crossing-form-degenerate}
Let $N:[0,T] \to \Sp(2)$ be the path pointwise defined by 
\[
	 		N(t) \:= \begin{pmatrix}
	 					1 & f(t)\\ 
	 					0 & 1
	 						 				   \end{pmatrix},\qquad t \in [0, T]
\]
with either $f(t)=t$ or $f(t)=-t$.
We denote by  $n$  the induced path of Lagrangian subspaces in $\R^4$ pointwise defined by $n(t) \:= \Gr\big(N(t)\big)$. Then we have 
\[
			 	\iCLM(\Delta, n(t), t \in  [0, T])=
			 \begin{cases}	
			 	0 & \textrm{ if } f(t)=t\\
			 	1 & \textrm{ if } f(t)=-t.
	 	\end{cases}
	 	\]
Moreover 
\[
\iCZ(N(t), t \in  [0, T])=
			 \begin{cases}	
			 	-1 & \textrm{ if } f(t)=t\\
			 	0 & \textrm{ if } f(t)=-t.
	 	\end{cases}
	 	\]
\end{lem}
\begin{figure}[hh]
	\centering
	\subfloat[][$\iCZ=0$.]{\includegraphics[width=0.45\textwidth]{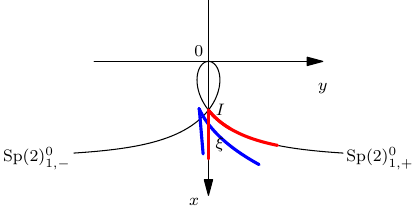}} \qquad
	\subfloat[][$\iCZ=-1$.]{\includegraphics[width=0.45\textwidth]{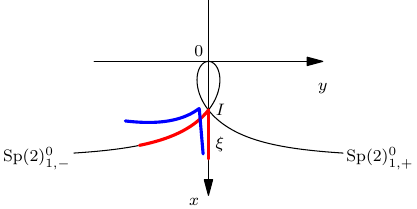}}
		\caption{c} \label{fig:Nalpha-nuovo}
\end{figure}

\begin{proof}
By the well definedness of the $\iCLM$-index, (cfr.\cite[pag. 138]{CLM03}), it is enough to compute the $\iCLM$-index for the path $t \mapsto n_\eps(t)$ where $n_\eps(t)=\Gr(N_\eps(t))$ for 
	 		\begin{equation}\label{eq:nepsilonalpha}
	 		 N_{\varepsilon}(t):= \begin{pmatrix}
	 			\cos\eps & f(t)\cos\eps+\sin\eps \\
	 			-\sin\eps & - f(t) \sin\eps+
	 			 \cos\eps 
	 			\end{pmatrix}
	 		\end{equation}
	 			 	The crossing instants are the zeros of the equation: 
	 		\begin{equation}\label{eq:crossingn1}
	 		 2 - 2\cos\eps + f(t)\sin\eps = 0.
	 		\end{equation} 		 
		 	We let $g_\eps(t)\:= 2 - 2\cos\eps + f(t)\sin\eps$ and we observe that $g_\eps(0)= 2-2\cos\eps>0$ for a positive sufficiently small $\eps$. Moreover $g_\eps(\tau)=0$ iff 
	 	\[
	 	t_\eps\:=\begin{cases}
	\dfrac{2(1-\cos \eps)}{\sin\eps}>0 & \textrm{ if } f(t)=-t\\
	\dfrac{2(\cos \eps-1)}{\sin\eps}<0 & \textrm{ if } f(t)=t
\end{cases}.
	 	\]	 	
	 	By this it directly follows that in the case $f(t)=t$, the (perturbed) path $t \mapsto n_\eps(t)$ has no crossing instants in the interval $[0,T]$.\footnote{
	 	 More generally if $f(t)>0$ for every $t \in [\alpha, \beta] \subseteq [0,T]$, then the perturbed path has no crossing instants in $[\alpha, \beta]$. In fact the equation:
	 	\[
	 	f(t_\eps)=\dfrac{2(1-\cos \eps)}{\sin\eps},\qquad \forall \, t \in [\alpha,\beta]
	 	\]
	 	 has no solution as the left hand side is  negative whilst  the right hand side is positive. }
	 	By this, immediately follows that 
	 	\[
	 		\iCLM(\Delta, n_\eps(\tau), \tau \in [0,T])=0.
			 	\]
We let us now compute the $\iCLM$-index for $f(t)=-t$.  We observe that we are in a very degenerate situation,  since  for every $t \in [0,T]$ the path $t \mapsto n_\eps(t)$ is contained in $\Sp^0(2)$, although it is not entirely contained in a fixed stratum, as the starting point is $N(0)=\Id$. Thus by the stratum homotopy invariance of the  $\iCLM$-index, we immediately get that 
\[
\iCLM(\Delta, n_\eps(t), t \in  [\delta, T])=0, \qquad \forall\  \delta \in (0, T]
\]
and hence
\begin{equation}\label{eq:poi-finisco}
\iCLM(\Delta, n(t), t \in  [\delta, T])=\Coindex \Gamma\big(n(t),\Delta, 0\big).
\end{equation}
By Lemma \ref{thm:crossing-form}, we get that 
\begin{equation}\label{eq:term-of-crossing-form}
\eta^{\prime}(t_{\varepsilon})=-\dfrac{1}{d(t_{\varepsilon})}d^{\prime}(\tep)y_0, \quad \xi'(\tep)=b'(\tep)y_0-\dfrac{b(\tep)}{d(\tep)}d'(\tep)y_0.
\end{equation}
Since $v:=\trasp{(x_0,y_0)}\in\ker(N_\epsilon(\tep)-\Id)$, then we have $x_0=\frac{b(\tep)}{1-\cos\epsilon}y_0$. By using Lemma \ref{thm:crossing-form} once more, we infer that 
\begin{equation}\begin{split}
\Gamma(n_\eps,\Delta,\tep)(v)&=-x_0\eta'(\tep)-y_0\xi'(\tep)\\
&=\dfrac{b(\tep)}{d(\tep)}\dfrac{d'(\tep)}{1-\cos\epsilon}y_0^2-b'(\tep)y_0^2+\dfrac{b(\tep)}{d(\tep)}d'(\tep)y_0^2\\&=\left[\dfrac{b(\tep)}{d(\tep)}\left(\dfrac{d'(\tep)}{1-\cos\epsilon}+d'(\tep)\right)-b'(\tep)\right]y_0^2.
\end{split}\end{equation}
The crossing form $\Gamma(l,\Delta,\tep)$ is a non-degenerate quadratic form on a one-dimensional vector space. In order to determine its inertia indexes, it is enough to know the sign of 
\[
C(\Gamma)\:=\dfrac{b(\tep)}{d(\tep)}\left(\dfrac{d'(\tep)}{1-\cos\varepsilon}+d'(\tep)\right)-b'(\tep).
\]
Since $\tep=\frac{2(1-\cos\varepsilon)}{\sin\varepsilon}$ then we get 
\[
 b'(\tep)=-\cos\varepsilon, d'(\tep)=\sin\varepsilon,  b(\tep)=-\tep\cos\varepsilon+\sin\varepsilon, d(\tep)=\tep\sin\varepsilon+\cos\epsilon.
 \]
 Thus 
  \begin{equation}\begin{split}
   C(\Gamma)&=\dfrac{-\tep\cos\varepsilon+\sin\varepsilon}{\tep\sin\varepsilon+\cos\varepsilon}(\dfrac{\sin\varepsilon}{1-\cos\varepsilon}+\sin\varepsilon)+\cos\varepsilon\\&=\dfrac{(1-\cos\varepsilon)^2}{2-\cos\varepsilon\sin\varepsilon}(\dfrac{\sin\varepsilon}{1-\cos\varepsilon}+\sin\varepsilon)+\cos\varepsilon\\&=\dfrac{(1-\cos\varepsilon)\sin\varepsilon}{2-\cos\varepsilon\sin\varepsilon}+\dfrac{(1-\cos\varepsilon)^2\sin\varepsilon}{2-\cos\varepsilon\sin\varepsilon}+\cos\varepsilon.
   \end{split}\end{equation}
It's easy to check that,  for $\varepsilon\rightarrow0^+$, we get  $\frac{(1-\cos\varepsilon)\sin\varepsilon}{2-\cos\varepsilon\sin\varepsilon}\rightarrow0$ and $\frac{(1-\cos\varepsilon)^2\sin\varepsilon}{2-\cos\varepsilon\sin\varepsilon}\rightarrow0$ . Therefore 
 $C(\Gamma)\sim 1$ holds for sufficient small $\varepsilon>0$. Hence the crossing form $C(\Gamma)$ is positive definite. Thus by the previous computation and by Equation \eqref{eq:poi-finisco}, we get 
  \begin{equation}
  \iCLM(n_\epsilon,\Delta,, t\in [0,T])=1.
  \end{equation}
This concludes the proof. 
\end{proof}
\begin{rem}
A different proof of Lemma \ref{thm:crossing-form-degenerate} could be as follows.  By the localization axiom of the Maslov index \cite[Lemma 5.2 (Shear property) ]{Gut14}, it follows that 
\[
\iRS(n(t), \Delta, t \in [0, T])= 
\begin{cases}
-\dfrac12  & \textrm{ if } f(t)=t\\
\dfrac12 & \textrm{ if } f(t)=-t.
\end{cases}
\]
Now, by \cite[Equation (3.7) pag.97]{LZ00}, it follows that 
\[
\iCLM(\Delta, n(t), t \in [0, T])= \iRS(n(t), \Delta, t \in [0, T])+ \dfrac12, 
\]
we can conclude that 
\begin{equation}\label{eq:finaleee}
\iCLM(\Delta, n(t), t \in [0, T])= 
\begin{cases}
0  & \textrm{ if } f(t)=t\\
1 & \textrm{ if } f(t)=-t.
\end{cases}
\end{equation}
The second claim  follows by Equation \eqref{eq:finaleee}  and Proposition \ref{thm:chiave}. This concludes the proof. 	 	 	
\end{rem}
\begin{rem}
A different proof of Lemma \ref{thm:crossing-form-degenerate}, without using any perturbation could be conceived even by using crossing forms. However, the reader should be aware on the fact that the case of symplectic shear is degenerate and a priori the formula for computing the $\iCLM$ index through crossing forms, is not available in that form. 
A different proof of Lemma \ref{thm:crossing-form-degenerate} goes as follows.  We observe that we are in a very degenerate situation,  since  for every $t \in [0,T]$ the path $t \mapsto N(t)$ is contained in $\Sp^0(2)$ even though it is not entirely contained in a fixed stratum, because of the starting point being in fact $N(0)=\Id$. Thus by the stratum homotopy invariant of the  $\iCLM$-index, we immediately get that 
\[
\iCLM(\Delta, n(t), t \in  [\eps, T])=0, \qquad \forall\  \eps \in (0, T]
\]
and hence
\[
\iCLM(\Delta, n(t), t \in  [\eps, T])=\Coindex \Gamma\big(n(t),\Delta, 0\big).
\]
By a   direct computation of the crossing form, we get 
		\begin{equation*}\begin{split}
	 	&a(t)=d(t)= 1,\quad b(t)=f(t), \quad c(t)= 0. \quad \textrm{ Moreover, }\\
	 	&a'(t)= d'(t)=c'(t)=0, \qquad b'(t)=f'(t).
	 	\end{split}\end{equation*}
	 	Thus $\xi'(0)=f'(0)y_0$ and 
	 	$\eta'(0)=0$.  
		Using Equations~\eqref{eq:lederivate} we get
	 		\begin{equation}\label{eq:crossingrotfin3}
	 			\Gamma(n,\Delta,0)[v] = -x_0 \eta'(t_*)-y_0 \xi'(t_*) = - f'(0)y_0^2.
		 	\end{equation}
		 	In particular for such a crossing form, 
		 	\begin{itemize}
		 	\item if $f(t)=t$, we get that $\Coindex \Gamma\big(n(t),\Delta, 0\big) =0$
			 				
			 	\item if $f(t)=-t$, we get $\Coindex \Gamma\big(n(t),\Delta, 0\big) =1$		
			 	 	\end{itemize}	
As the crossing form is degenerate, the conclusion follows by using the generalization through partial signatures as given by authors in \cite[Proposition 2.1]{GPP04}	. 	

The second claim  follows by the first one and Proposition \ref{thm:chiave}. This concludes the proof. 	 	 	
\end{rem}
By Lemma \ref{thm:crossing-form-degenerate} and by the well-posedness of the $\iCLM$-index, the following is clear.
\begin{cor}\label{thm:finalmente}
	Let $N:[0,T] \to \Sp(2)$  be the symplectic path defined in Lemma \ref{thm:crossing-form-degenerate} and $N_\eps$ pointwise defined by $N_\eps(t)\:=e^{-\eps J}N(t)$. Thus we have: 
\[
\iCZ(N_\eps(t), t \in  [0, T])=
			 \begin{cases}	
			 	-1 & \textrm{ if } f(t)=t\\
			 	0 & \textrm{ if } f(t)=-t.
	 	\end{cases}
	 	\]
\end{cor}

\section{Indices and stability of Keplerian orbits}\label{sec:linear-stability}

The aim of this section is to compute the Conley-Zehnder index of a Keplerian ellipses with eccentricity $e \in [0,1)$. In Subsection~\ref{subsec:CZ-polar} we explicitly compute the $\iCZ$-index in the case of circular Keplerian orbit and finally we conclude the general case of Keplerian ellipses having eccentricity $e \in [0,1)$.  


\subsection{The Conley-Zehnder index in polar coordinates} \label{subsec:CZ-polar}

Let us now come back to the Lagrangian function given at Equation~\ref{eq:EL-polar-Kepler} whose induced Hamiltonian function is given by 
\begin{equation}\label{eq:Hamiltonian-Kepler-polar}
	H\big(r,\vartheta, \dot r, \dot \vartheta\big)= \dfrac12 \mu \big(\dot r^2+ r^2\dot \vartheta^2 \big)- U\big(r\big)=\dfrac12\left[\dfrac{p_r^2}{\mu}+\dfrac{p_\vartheta^2}{r^2 \mu}\right]-\dfrac{m}{r}
\end{equation}
where $(p_r, p_\vartheta)=\big(\mu \dot r, \mu r^2 \dot \vartheta\big)$. The induced Hamiltonian system is given by 
\begin{equation}\label{eq:HS-polar-kepler}
\begin{cases}
	\dot p_r= -\partial_r H (p_r, p_\vartheta, r, \vartheta) =-\dfrac{m}{r^2}+\dfrac{p_\vartheta^2}{\mu r^3}\\
	\dot p_\vartheta= -\partial_\vartheta H(p_r, p_\vartheta, r, \vartheta)=0\\
	\dot r= \partial_{p_r} H(p_r, p_\vartheta, r, \vartheta)= \dfrac{p_r}{\mu}\\ 
	\dot \vartheta= \partial_{p_\vartheta} H(p_r, p_\vartheta, r, \vartheta)=\dfrac{p_\vartheta}{r^2\mu}
\end{cases}
\end{equation}
By linearizing the Hamiltonian system given in Equation~\eqref{eq:HS-polar-kepler} at the circular solution $r(t)=re^{i\omega t}$ for $t \in [0, 2\pi/\omega]$,  then we get 
\begin{equation}\label{eq:linearization-HS-kepler}
	\begin{cases}
	\dot y_r= \left[\dfrac{2m}{r^3}-\dfrac{3 p_\vartheta^2}{\mu r^4}\right]x_r+\dfrac{2p_\vartheta}{\mu r^3}y_\vartheta\\
	\dot y_\vartheta=0\\
	\dot x_r= \dfrac{y_r}{\mu}\\ 
	\dot x_\vartheta= -\dfrac{2p_\vartheta}{r^3\mu}x_r+  \dfrac{y_\vartheta}{r^2\mu}.
\end{cases}
\end{equation}
Setting $w=\trasp{(y_r,y_\vartheta,x_r, x_\vartheta)}$, then the linearized Hamiltonian system given at  Equation~\eqref{eq:linearization-HS-kepler} can be written as $\dot w= L w$ where $L$ is the four  by four matrix given by 
\begin{equation}
L\:=\begin{bmatrix}
	0 & \dfrac{2p_\vartheta}{\mu r^3} & \left[\dfrac{2m}{r^3}-\dfrac{3 p_\vartheta^2}{\mu r^4}\right] & 0 \\
	0 & 0 & 0 & 0\\
	\dfrac{1}{\mu} & 0 & 0 & 0\\
	0 & \dfrac{1}{r^2\mu} &-\dfrac{2p_\vartheta}{r^3\mu} &0
\end{bmatrix}= \begin{bmatrix}
 	0 & A & C & 0\\
 	0 & 0 & 0 & 0\\
 	D & 0 & 0 & 0 \\
 	0 & B & -A & 0
 \end{bmatrix}
\end{equation}
where  
\[
A\:= \dfrac{2p_\vartheta}{\mu r^3},\quad  B\:=\dfrac{1}{r^2\mu},\quad  C\:=\left[\dfrac{2m}{r^3}-\dfrac{3 p_\vartheta^2}{\mu r^4}\right] \ \textrm{ and finally} \quad D\:=\dfrac{1}{\mu}.
\]
It is worthwhile to observe that, the matrix $L$ is a time independent Hamiltonian matrix. Thus, by setting,   $-CD= \omega^2$, the fundamental (matrix) solution is given by $\phi_0(t)= e^{L t}$ where $t \in \left[0, \dfrac{2\pi}{\omega}\right]$. Now, since the determinant of $L-\lambda \Id_4$ is $\lambda^2(\lambda^2+\omega^2)$, there exists a symplectic matrix $P \in \Sp(4)$ such that 
\[
L= P \left(\begin{bmatrix}
	0 & s(r)\\ 0 & 0 
\end{bmatrix}\diamond \begin{bmatrix}
 	0 & -\omega\\
 	\omega & 0
 \end{bmatrix}
\right) P^{-1}
\]
where $r \mapsto s(r)$ is a positive function (to be determined) and by using the direct sum property of the Conley-Zehnder index (cf. Section~\ref{sec:Preliminaries-on-Maslov-and-Conley-Zehnder-indexes}), then we get that 
\[
\iCZ\big(\phi_0(t), t \in [0, T]\big)= \iCZ\big(\overline \phi_0(t), t \in [0, T]\big)+1
\]
where $\overline \phi(t)= \begin{bmatrix}
 1& s(r) t\\ 0 & 1	 \end{bmatrix}$, and $ t \in [0,T]$ and $T=2\pi/\omega$. In order to compute the function $s(r)$ (actually,  
 by using Lemma~\ref{thm:crossing-form-degenerate}, we only need to compute the sign of this function),  we proceed as follows. Denoting by $\mathscr E=\{e_1,e_2, e_3,e_4\}$ the canonical basis of $\R^4$ we start to observe that $L e_4=0$. Moreover, setting $v=\trasp{(0,-1,A/C,0)}$, then we get that 
 \[
 L v=\left(\dfrac{A^2 }{C}-B\right)e_4=\left(\dfrac{A^2 D}{\omega^2}-B\right)e_4= -\dfrac{1}{\mu r^2}\left[1+ \dfrac{4p^2_\vartheta}{2m\mu r-3 p^2_\vartheta}\right] e_4.
 \]
Since for circular motions, by the discussion performed at the end of Section~\ref{sec:variational} and more precisely at Equation~\eqref{eq:semilatusrectum}, we get that $r=r_0= k^2/(\mu m)= p_\vartheta^2/(\mu m)$ were $k$ denotes the angular momentum,  we finally get that 
\[
Lv= \dfrac{3}{\mu r^2} e_4.
\]
We observe that, being $\omega_0(e_4, v)=1$, it readily follows that $\{e_4, v\}$ is a symplectic basis of that invariant subspace and the function $r \mapsto s(r)=\dfrac{3}{\mu r^2}$ is positive. In particular, by using Lemma~\ref{thm:crossing-form-degenerate},  we get that 
\[
\iCZ\big(\overline \phi_0(t), t \in [0, T]\big)=-1.
\]
Summing up the previous computation we finally get the following result. 
\begin{prop}
	The Conley-Zehnder index of a planar circular solution of Equation~\eqref{eq:kepler-one-center} on a prime period, vanishes.
\end{prop}
\begin{rem}
	In the case of homogeneous singular potentials an analogous result was already proved by authors in \cite{BJP16}.  A similar result in the case of Keplerian orbits on constant curvature surfaces has been proved in \cite{DDZ19}.
\end{rem}
\begin{prop}\label{thm:lemma3.3}
	Let $\phi_e$ be the monodromy matrix of the Keplerian ellipse having  period $T$ and eccentricity $e$. Then 
	there exists $P \in \Sp(4)$  such that 
	\[
	M= P^{-1}\big(N_1(1,1) \diamond \Id_2\big)P, 
	\]
	where $N_1(1,1)\:=\begin{pmatrix}1 & 1\\ 0& 1	\end{pmatrix} $.
\end{prop}
\begin{proof}
By 	\cite[Lemma 3.3]{HS10} we get that there exists $P \in \Sp(4)$  such that 
	\[
	M= P^{-1}\big(N_1(1,1) \diamond M_1\big)P
	\]
	where the $2\times 2$ Jordan block $N_1(1,1)$ corresponds to the energy integral. By the integrability of the Kepler problem, we know that the angular momentum is a first integral. This implies that $\sigma(M_1)=\{1\}$. By the basic normal form of  a $2\times 2$ symplectic matrix, $M_1$ has to be symplectically similar to a matrix of the form $N_1(1,b)=\begin{bmatrix}1&b\\0& 1	 \end{bmatrix}$ where $b\in \{0, 1, -1\}$. Now,  we observe that in the negative energy $3$-dimensional hypersurface of the phase space, every solution is an elliptic orbit having prime period $T$. So the time $T$ fundamental solution restricted to the fixed negative energy hypersurface $\Sigma_h$ corresponding to the energy level $h$  (which is a $\phi_e(t)$-invariant manifold) has to be the identity map (otherwise $\Sigma_h$ would not be invariant). By this argument, we directly conclude that 
	\[
	\dim \ker (M-\Id_4) =3
	\]  
	and so $M$ is symplectically similar to 
	\begin{equation}\label{eq:fund-sol-e} 
		N_1(1,1) \diamond \Id_2.
		\end{equation}
		 	This concludes the proof.
\end{proof}
  
\begin{thm}\label{thm:maslov-kepler-iterate-1}
Let 	$\gamma$ be a Keplerian ellipse (i.e. a solution of Equation~\eqref{eq:kepler-one-center} with  prime period $T$) and let $\gamma^k$ its $k$-th iteration. Then, we have 
\[
\iMor(\gamma^k)=2(k-1)\qquad \forall\, k \in \N^*
\]
where, as before, we denoted by $\iMor(\gamma)$ the Morse index of $\gamma$. In particular 
\[
\iMor(\gamma)=0 \quad \textrm{ and } \lim_{k \to +\infty} \dfrac{\iMor(\gamma^k)}{k}=2.
\]
 \end{thm}
\begin{proof}
Since the time rescaling doesn't change the Morse index, by Proposition \ref{thm:indextheorem}, we get 
\[
\iMor(\gamma^k)= \iCZ(\phi_{e}(t), t \in [0, Tk]).
\]
By invoking Equation~\eqref{eq:fund-sol-e}  and as direct consequence of Lemma~\ref{thm:crossing-form-degenerate} and Lemma~\ref{thm:Gutt41} and by using the additivity property of the Conley-Zehnder index under concatenation of paths, we finally get that 
\begin{equation}\label{eq:sonost-1}
 \iCZ(\phi_{e}(t), t \in [0, T k])=  2k-2
 \end{equation}
To conclude, we observe that since the Conley-Zehnder index of the fundamental solution $\phi_e$ only depends on the monodromy matrix, the thesis follows by previous computation and by Equation~\eqref{eq:fund-sol-e}.  This concludes the proof of the first part.  The second follows straightforward. 
\end{proof}
\begin{rem}
It is worth noticing that the contribution given by the conservation law of the energy to the $\iCZ$-index is $-1$ and the symplectic normal form is given by the $2 \times 2$ Jordan block relative to the eigenvalue $1$ whereas the symplectic normal form corresponding to the  conservation law of the angular momentum is given by the $2 \times 2$ identity matrix. As expected we are in a very degenerate situation. 
\end{rem}
We conclude the section by summarizing the  stability properties of the Keplerian ellipses. 
\begin{thm}\label{thm:stability}
Let $\gamma$ be a Keplerian ellipse. Then it is elliptic, meaning that all the eigenvalues belongs to $\U$. Moreover it is spectrally stable and not linearly stable. 
\end{thm}
\begin{proof}
Denoting by $M$ the monodromy matrix, it follows that there exists $P \in \Sp(4)$  such that 
\begin{equation}\label{eq:monodromy}
	M= P^{-1}\left(\begin{bmatrix}
		1 & 1\\
		0 & 1
	\end{bmatrix} \diamond \begin{bmatrix}
		1&0 \\
		0 & 1
	\end{bmatrix}\right)P.
\end{equation}
In particular  $\sigma(M)=\{1\} \in \U$  and the  algebraic multiplicity of its (unique) Floquet multiplier is $4$. This proves the first claim. The second follows by observing that 
 $M$ is not diagonalizable (having a non-trivial Jordan block); thus  in particular $\gamma$ is spectrally and not linearly stable.  We also observe that  its  nullity is  $3= \dim \ker (M-\Id)$. 
\end{proof}

\vspace{1cm}
\noindent
\textsc{Henry Kavle}\\
Department of Mathematics and Statistics\\
Queen's University, Kingston (ON)\\
K7K 3N6 Kingston (Ontario) \\
Canada\\
E-mail:\texttt{kavle.h@queensu.ca}

\vspace{1cm}
\noindent
\textsc{Prof. Daniel Offin}\\
Department of Mathematics and Statistics\\
Queen's University, Kingston (ON)\\
K7K 3N6 Kingston (Ontario) \\
Canada\\
E-mail:\texttt{offind@queensu.ca}

\vspace{1cm}
\noindent
\textsc{Prof. Alessandro Portaluri}\\
DISAFA\\
Università degli Studi di Torino\\
Largo Paolo Braccini, 2 \\
10095 Grugliasco, Torino\\
Italy\\
Website: \texttt{https://sites.google.com/view/alessandro-portaluri/home}\\
E-mail: \texttt{alessandro.portaluri@unito.it}

\end{document}